\documentclass[11pt, a4paper]{article}

\usepackage{amsmath, epsfig, url}
\usepackage{latexsym}
\usepackage{amssymb}
\usepackage{latexsym}
\usepackage{graphicx}
\usepackage{url}

\newtheorem{theorem}{Theorem}
\newtheorem{lemma}[theorem]{Lemma}
\newtheorem{corollary}[theorem]{Corollary}

\newtheorem{construction}[theorem]{Construction}

\newtheorem{remark}{Remark}

\input amssym.def
\input amssym.tex

\long\def\delete#1{}

\usepackage{color}

\definecolor{Thistle}{rgb}{0.847,0.749,0.847}
\definecolor{Khaki}{rgb}{0.941,0.902,0.549}
\definecolor{Orchid}{rgb}{0.855,0.439,0.839}
\definecolor{MediumOrchid}{rgb}{0.729,0.333,0.827}

\definecolor{brown}{rgb}{0.8,0.5,0}
\definecolor{LightBrown}{rgb}{0.8,0.2,0.4}

\definecolor{DarkGray}{rgb}{0.78,0.78,0.78}
\definecolor{DarkMidGray}{rgb}{0.81,0.81,0.81}
\definecolor{MidGray}{rgb}{0.85,0.85,0.85}
\definecolor{LightGray}{rgb}{0.88,0.88,0.88}
\definecolor{VeryLightGray}{rgb}{0.96,0.96,0.96}

\definecolor{GrayA}{rgb}{0.7,0.7,0.7}
\definecolor{GrayB}{rgb}{0.78,0.78,0.78}
\definecolor{GrayC}{rgb}{0.80,0.80,0.80}
\definecolor{GrayD}{rgb}{0.82,0.82,0.82}
\definecolor{GrayE}{rgb}{0.84,0.84,0.84}
\definecolor{GrayF}{rgb}{0.86,0.86,0.86}
\definecolor{GrayG}{rgb}{0.88,0.88,0.88}
\definecolor{GrayH}{rgb}{0.90,0.90,0.90}
\definecolor{GrayI}{rgb}{0.92,0.92,0.92}
\definecolor{GrayJ}{rgb}{0.94,0.94,0.94}

\definecolor{VeryLightBlue}{rgb}{0.9,0.9,1}
\definecolor{LightBlue}{rgb}{0.8,0.8,1}
\definecolor{MidBlue}{rgb}{0.5,0.5,1}
\definecolor{DarkBlue}{rgb}{0,0,0.6}

\definecolor{Gold}{rgb}{1,0.843,0}

\definecolor{LightGreen}{rgb}{0.88,1,0.88}
\definecolor{MidGreen}{rgb}{0.6,1,0.6}
\definecolor{DarkGreen}{rgb}{0,0.6,0}

\definecolor{VeryLightYellow}{rgb}{1,1,0.9}
\definecolor{LightYellow}{rgb}{1,1,0.6}
\definecolor{MidYellow}{rgb}{1,1,0.5}
\definecolor{DarkYellow}{rgb}{1,1,0.2}

\definecolor{VeryLightRed}{rgb}{1,0.9,0.9}
\definecolor{LightRed}{rgb}{1,0.8,0.8}
\definecolor{MidRed}{rgb}{1,0.55,0.55}

\def\choose#1#2{\left (\!\!\begin{array}{c}#1\\#2\end{array}\!\!\right )}

\def\qed{\hfill$\Box$\vspace{12pt}}

\def\BB{{\cal B}}

\def\DD{{\cal D}}

\def\LL{{\cal L}}

\def\PP{{\cal P}}

\def\I{{\rm I}}

\def\bfL{{\bf L}}

\def\b0{{\bf 0}}

\def\De{\Delta}
\def\Ga{\Gamma}

\def\Si{\Sigma}
\def\Om{\Omega}

\def\a{\alpha}
\def\b{\beta}

\def\g{\gamma}
\def\l{\lambda}
\def\s{\sigma}
\def\t{\tau}
\def\ve{\varepsilon}

\def\Aut{{\rm Aut}}

\def\PSL{{\rm PSL}}
\def\PGL{{\rm PGL}}

\def\AGL{{\rm AGL}}
\def\PGL{{\rm PGL}}
\def\AG{{\rm AG}}
\def\PG{{\rm PG}}
\def\GF{{\rm GF}}

\def\AGammaL{{\rm A\Gamma L}}
\def\PGammaL{{\rm P\Gamma L}}

\def\Sp{{\rm Sp}}

\def\Arc{{\rm Arc}}

\def\val{{\rm val}}

\newcommand{\pmat}[1]{\begin{pmatrix}#1\end{pmatrix}}

\title{Symmetric graphs with 2-arc transitive quotients}
\author{Guangjun Xu and Sanming Zhou\\ \\
{\small
Department of Mathematics and Statistics}\\
{\small The University of Melbourne}\\
{\small Parkville, VIC 3010, Australia}\\
{\small E-mail: {\it \{gx, smzhou\}@ms.unimelb.edu.au}}}

\date{}

\begin{document}

\openup 0.5\jot\maketitle

\vspace{-1cm}

\smallskip
\begin{abstract}
A graph $\Ga$ is $G$-symmetric if $\Ga$ admits $G$ as a group of automorphisms acting transitively on the set of vertices and the set of arcs of $\Ga$, where an arc is an ordered pair of adjacent vertices. In the case when $G$ is imprimitive on $V(\Ga)$, namely when $V(\Ga)$ admits a nontrivial $G$-invariant partition $\BB$, the quotient graph $\Ga_{\BB}$ of $\Ga$ with respect to $\BB$ is always $G$-symmetric and sometimes even $(G, 2)$-arc transitive. (A $G$-symmetric graph is $(G, 2)$-arc transitive if $G$ is transitive on the set of oriented paths of length two.) In this paper we obtain necessary conditions for $\Ga_{\BB}$ to be $(G, 2)$-arc transitive (regardless of whether $\Ga$ is $(G, 2)$-arc transitive) in the case when $v-k$ is an odd prime $p$, where $v$ is the block size of $\BB$ and $k$ is the number of vertices in a block having neighbours in a fixed adjacent block. These conditions are given in terms of $v, k$ and two other parameters with respect to $(\Ga, \BB)$ together with a certain 2-point transitive block design induced by $(\Ga, \BB)$. We prove further that if $p=3$ or $5$ then these necessary conditions are essentially sufficient for $\Ga_{\BB}$ to be $(G, 2)$-arc transitive.
 

\end{abstract}

\section{Introduction}
A graph $\Ga=(V(\Ga),E(\Ga))$ is {\em $G$-symmetric} if $\Ga$ admits $G$ as a group of automorphisms such that $G$ is transitive on $V(\Ga)$ and on the set of arcs of $\Ga$, where an {\em arc} is an ordered pair of adjacent vertices. If in addition $\Ga$ admits a {\em nontrivial $G$-invariant partition}, that is, a partition $\BB$ of $V(\Ga)$ such that $1< |B| < |V(\Ga)|$ and $B^g := \{\a^g: \a \in B\} \in \BB$ for any $B \in \BB$ and $g \in G$ (where $\a^g$ is the image of $\a$ under $g$), then $\Ga$ is called an {\em imprimitive $G$-symmetric graph}. In this case the {\em quotient graph} $\Ga_{\BB}$ of $\Ga$ with respect to $\BB$ is defined to have vertex set $\BB$ such that $B, C \in \BB$ are adjacent if and only if there exists at least one edge of $\Ga$ between $B$ and $C$. It is readily seen that $\Ga_{\BB}$ is $G$-symmetric under the induced action of $G$ on $\BB$. We assume that $\Ga_{\BB}$ contains at least one edge, so that each block of $\BB$ is an independent set of $\Ga$. Denote by $\Ga(\alpha)$ the neighbourhood of $\alpha \in V(\Ga)$ in $\Ga$, and define $\Ga(B)=\cup_{\alpha \in B}\Ga(\alpha)$ for $B \in \BB$. For blocks $B, C \in \BB$ adjacent in $\Ga_{\BB}$, let $\Ga[B,C]$ be the bipartite subgraph of $\Ga$ induced by $(B \cap \Ga(C)) \cup (C \cap \Ga(B))$. Since $\Ga_{\BB}$ is $G$-symmetric, up to isomorphism $\Ga[B,C]$ is independent of the choice of $(B, C)$. Define $\Ga_{\BB}(\a) := \{C \in \BB: \Ga(\a) \cap C \ne \emptyset\}$ and $\Ga_{\BB}(B) := \{C \in \BB: \mbox{$B$ and $C$ are adjacent in $\Ga_{\BB}$}\}$, the latter being the neighborhood of $B$ in $\Ga_{\BB}$. 
Define 
$$
v:=|B|,\,\; k:=|B \cap \Ga(C)|,\,\; r := |\Ga_{\BB}(\a)|,\,\; b := \val(\Gamma_{\mathcal {B}})
$$ 
to be the block size of $\BB$, the size of each part of the bipartition of $\Ga[B,C]$, the number of blocks containing at least one neighbour of a given vertex, and the valency of $\Ga_{\BB}$, respectively. These parameters depend on $(\Ga, \BB)$ but are independent of $\a \in V(\Ga)$ and adjacent $B, C \in \BB$. 

In \cite{gp1} Gardiner and Praeger introduced a geometrical approach to imprimitive symmetric triples $(\Ga, G, \BB)$ which involves $\Ga_{\BB}$, $\Gamma[B, C]$ and an incidence structure $\DD(B)$ with point set $B$ and block set $\Gamma_{\mathcal {B}}(B)$. A `point' $\a \in B$ and a `block' $C \in \Gamma_{\mathcal {B}}(B)$ are incident in $\DD(B)$ if and only if $\a \in \Gamma(C)$; we call $(\a, C)$ a flag of $\DD(B)$ and write $\a \I C$. It is clear that $\mathcal{D}(B) = (B, \Gamma_{\mathcal {B}}(B), I)$ is a $1$-$(v, k, r)$ design \cite{gp1} with $b$ blocks which admits $G_B$ as a group of automorphisms acting transitively on its points, blocks and flags, where \emph{$G_B$} is the setwise stabilizer of $B$ in $G$. Note that $vr = bk$. Define $\overline{\DD}(B) := (B, \Ga_{\BB}(B), \overline{\I})$ to be the complementary structure \cite{lz} of $\DD(B)$ for which $\a \overline{\I} C$ if and only if $\a \not \in \Ga(C)$. Then $\overline{\DD}(B)$ is $1$-$(v, v-k, b-r)$ design with $b$ blocks. Up to isomorphism $\mathcal{D}(B)$ and $\overline{\DD}(B)$ are independent of $B$. The cardinality of $\{D \in \Ga_{\BB}(B): \Ga(D) \cap B = \Ga(C) \cap B\}$, denoted by $m$, is independent of the choice of adjacent $B, C \in \BB$ and is called the {\em multiplicity} of $\DD(B)$.
 
An {\em $s$-arc} of $\Ga$ is a sequence $(\a_0, \a_1, \ldots, \a_s)$ of $s+1$ vertices of $\Ga$ such that $\a_{i}, \a_{i+1}$ are adjacent for $i = 0, \ldots , s - 1$ and $\a_{i-1} \neq \a_{i+1}$ for $i = 1, \ldots, s - 1$. If $\Ga$ admits $G$ as a group of automorphisms such that $G$ is transitive on $V(\Ga)$ and on the set of $s$-arcs of $\Ga$, then $\Ga$ is called {\em $(G,s)$-arc transitive} \cite{b}. A $(G, 1)$-arc transitive graph is precisely a $G$-symmetric graph, and a $(G,s)$-arc transitive graph is $(G,s-1)$-arc transitive. 

This paper was motivated by the following questions asked in \cite{mpz}: When does a quotient of a symmetric graph admit a natural 2-arc transitive group action? If there is such a quotient, what information does this give us about the original graph? These questions were studied in \cite{mpz, jlw, lpz, lpz2, xz, zhou98, zhou2008}, with a focus on the case where $v-k \ge 1$ or $k \ge 1$ is small. In the present paper we consider the more general case where $k=v-p$ for a prime $p \ge 3$. In this case we obtain necessary conditions for $\Ga_{\BB}$ to be $(G, 2)$-arc transitive, regardless of whether $\Ga$ is $(G, 2)$-arc transitive. We prove further that when $p=3$ or $5$ such necessary conditions are essentially sufficient for $\Ga_{\BB}$ to be $(G, 2)$-arc transitive. 

A few definitions and notations are needed before stating our main result. Let $G$ and $H$ be groups acting on $\Om$ and $\Lambda$ respectively. The action of $G$ on $\Om$ is said to be {\em permutationally isomorphic} \cite[p.17]{Dixon-Mortimer} to the action of $H$ on $\Lambda$ if there exist a bijection $\rho: \Om \rightarrow \Lambda$ and a group isomorphism $\psi: G \rightarrow H$ such that $\rho(\a^g) = (\rho(\a))^{\psi(g)}$ for all $\a \in \Om$ and $g \in G$. In the case when $G=H$ and the actions of $G$ on $\Om$ and $\Lambda$ are permutationally isomorphic, we simply write $G^{\Om} \cong G^{\Lambda}$. 

Now we return to our discussion on imprimitive symmetric triples $(\Ga, G, \BB)$. 
Define $G_{(B)} = \{g \in G_B: \a^g = \a\;\mbox{for every $\a \in B$}\}$ to be the pointwise stabilizer of $B$ in $G$, and $G_{[B]} = \{g \in G_B: C^g = C\;\mbox{for every $C \in \Ga_{\BB}(B)$}\}$ the pointwise stabilizer of $\Ga_{\BB}(B)$ in $G_B$. As usual, by $G_B^B$ we mean the group $G_B/G_{(B)}$ with its action restricted to $B$, and by $G_B^{\Ga_{\BB}(B)}$ we mean $G_B/G_{[B]}$ with its action restricted to $\Ga_{\BB}(B)$. (Thus, whenever we write $G_B^B \cong G_B^{\Ga_{\BB}(B)}$, we mean that the actions of $G_B$ on $B$ and $\Ga_{\BB}(B)$ are permutationally isomorphic.) Define $G_{(\BB)} = \{g \in G: B^g = B\;\mbox{for every $B \in \BB$}\}$.  

Let $\Si$ be a graph and $\De$ a subset of the set of 3-arcs of $\Si$. We say that $\De$ is {\em self-paired} if $(\t, \s, \s', \t') \in \De$ implies $(\t', \s', \s, \t) \in \De$. In this case the {\em 3-arc graph} $\Xi(\Si, \De)$ is defined \cite{lpz} to have arcs of $\Si$ as its vertices such that two such arcs $(\s, \t), (\s', \t')$ are adjacent if and only if $(\t, \s, \s', \t') \in \De$. We denote by $n \cdot \Si$ the graph which is $n$ vertex-disjoint copies of $\Si$, and by $C_n$ the cycle of length $n$. We may view the complete graph $K_n$ on $n$ vertices as a degenerate design of block size two. 

As shown in \cite{lz}, when $\Ga_{\BB}$ is $(G, 2)$-arc transitive, the {\em dual design} $\DD^{*}(B)$ of $\DD(B)$ plays a significant role in the study of $\Ga$, where $\DD^{*}(B)$ is obtained from $\DD(B)$ by interchanging the roles of points and blocks but retaining the incidence relation. Since in this case $G_B$ is 2-transitive on $\Ga_{\BB}(B)$, as observed in \cite{lz},
\begin{equation}
\label{eq:l}
\l:= |\Ga(C) \cap \Ga(D) \cap B|
\end{equation}
is independent of the choice of distinct $C, D \in \Ga_{\BB}(B)$. Denote by $\overline{\DD^*}(B)$ the complementary incidence structure of $\DD^{*}(B)$, which is defined to have the same `point' set $\Ga_{\BB}(B)$ as $\DD^{*}(B)$ such that a `point' $C \in \Ga_{\BB}(B)$ is incident with a `block' $\a \in B$ if and only if $C \not \in \Ga_{\BB}(\a)$.  
As observed in \cite[Theorem 3.2]{lz}, if $\Ga_{\BB}$ is $(G, 2)$-arc transitive, then either $\l=0$ or $\DD^{*}(B)$ is a 2-$(b, r, \l)$ design with $v$ blocks, and either $\overline{\l} := v - 2k + \l = 0$ or $\overline{\DD^*}(B)$ is a 2-$(b, b-r, \overline{\l})$ design with $v$ blocks. Moreover, each of $\DD^{*}(B)$ and $\overline{\DD^*}(B)$ admits \cite{lz} $G_B$ as a group of automorphisms acting 2-transitively on its point set and transitively on its block set. The first main result in this paper, Theorem \ref{thm:para} below, gives the parameters of $\DD^{*}(B)$ and information about $\Ga$, $\DD^{*}(B)$ or/and the action of $G_B$ on $\Ga_{\BB}(B)$ in the case when $k=v-p$ for a prime $p \ge 3$. Our proof of this result relies on the classification of finite 2-transitive groups (see e.g. \cite{Dixon-Mortimer}) and that of 2-transitive symmetric designs \cite{k} (which in turn rely on the classification of finite simple groups). Without loss of generality we may assume that $\Ga_{\BB}$ is connected.

\begin{theorem} 
\label{thm:para}
Let $\Ga$ be a $G$-symmetric graph with $V(\Ga)$ admitting a nontrivial $G$-invariant partition $\BB$ such that $k=v-p \geq 1$ and $\Ga_{\BB}$ is connected with valency $b \ge 2$, where $p \ge 3$ is a prime and $G \le \Aut(\Ga)$. Suppose $\Ga_{\BB}$ is $(G, 2)$-arc transitive. Then one of (a)-(f) in Table \ref{tab0} occurs, and in (c)-(f) the parameters of the 2-$(b, r, \l)$ design $\DD^{*}(B)$ with $v$ blocks are given in the third column of the table. 

Moreover, in (a), $\Ga \cong (|V(\Ga)|/2) \cdot K_2$, $G_B^B \cong G_B^{\Ga_{\BB}(B)}$ is 2-transitive of degree $p+1$, and any connected $(p+1)$-valent $(G,2)$-arc transitive graph can occur as $\Ga_{\BB}$ in (a). 

In (b), $\Ga \cong n \cdot \Ga[B, C]$ where $n=|V(\Ga)|/2p$, $\Ga_{\BB} \cong C_n$, and $G/G_{(\BB)} = D_{2n}$. 

In (c), $G_B^B \cong G_B^{\Ga_{\BB}(B)}$ is isomorphic to a 2-transitive subgroup of $\PGammaL(n+1,q)$, and $G$ is faithful on $\BB$.

In (d), $G_B^B \cong G_B^{\Ga_{\BB}(B)} \cong \PSL(2,11)$.

In (e), $V(\Ga)$ admits a $G$-invariant partition $\PP$ with block size $p$ which is a refinement of $\BB$ such that $\Ga_{\PP} \cong \Xi(\Ga_{\BB}, \Delta)$ for a self-paired $G$-orbit $\Delta$ on the set of 3-arcs of $\Ga_{\BB}$. Moreover, $\hat \BB = \{\hat{B}: B \in \BB\}$ (where $\hat B$ is the set of blocks of $\PP$ contained in $B$) is a $G$-invariant partition of $\PP$ such that $(\Ga_{\PP})_{\hat \BB} \cong \Ga_{\BB}$ and the parameters with respect to $(\Ga_{\PP}, \hat \BB)$ are given by $v_{\hat \BB} = b_{\hat \BB} = a$ and $k_{\hat \BB} = r_{\hat \BB} = a-1$. 

In (f), if $s=1, 2$, then all possibilities are given in Tables \ref{tab1}-\ref{tab2} respectively, where $G_B^{\Ga_{\BB}(B)}$ is isomorphic to the group or a 2-transitive subgroup of the group in the first column (with natural actions).  
\end{theorem}  

\begin{table}
\begin{center}
  \begin{tabular}{l|l|l|l}
\hline 
Case & $\overline{\DD^*}(B)$  & $(v,b,r,\l)$ & Conditions \\  \hline
(a) & & $(p+1,p+1,1,0)$ & \\  \hline
(b) & & $(2p,2,1,0)$ & \\  \hline
     & & & $p = \frac{q^{n}-1}{q-1}, n \ge 2$ \\
(c) & $\PG_{n-1}(n, q)$ & $\left(\frac{q^{n+1}-1}{q-1}, \frac{q^{n+1}-1}{q-1}, q^n, q^n - q^{n-1}\right)$ & \mbox{$q$ a prime power} \\
     & & & $\frac{q^{n}-1}{q-1}$ is a prime \\  \hline
(d) & 2-$(11, 5, 2)$ & $(11, 11, 6, 3)$ & $p=5$ \\  \hline
(e) & & $(pa,a,a-1,p(a-2))$ & $a \ge 3$ \\  \hline
     & & & $a \ge 2, s \ge 1$ \\
     & & & $a$ a divisor of $ps+1$ \\
(f) & & $\left(pa, ps+1, \frac{(ps+1)(a-1)}{a}, p(a-2) + \frac{ps-a+1}{as}\right)$ & 
$s$ a divisor of $\frac{ps-a+1}{a}$ \\
     & & & $\frac{a-1}{p-a} \le s \le a-1 \le p-2$ \\  \hline
\end{tabular}
  \caption{Theorem \ref{thm:para}.}
\label{tab0}
  \end{center}
\end{table}

\begin{table}
\begin{center}
  \begin{tabular}{l|l|c|l}
\hline
$G_B^{\Ga_{\BB}(B)}$ & $\DD^*(B)$  & $(v,b,r,\l)$ & Conditions \\  \hline
$A_{p+1}$ & $\overline{\DD^*}(B) \cong K_{p+1}$ & & $a = \frac{p+1}{2}$ \\ \hline
   & & & $1 \le m \le n-1$ \\ 
   & & & $p=2^n-1$ \\
   & & & a Mersenne prime \\
$\le \AGL(n, 2)$ & &
$\pmat{2^m (2^n - 1)\\ 2^n\\ 2^n - 2^{n-m}\\ (2^m - 1)(2^n - 2^{n-m} - 1)}$ & 
                                    $r^* = (2^n - 1)(2^m - 1)$ \\ \hline
$\le \PGL(2, p)$ & & & $a-1$ a divisor of $p-1$ \\ \hline
$\Sp_{4}(2)$ & $2$-$(6, 3, 2)$ & & $p=5$ \\ \hline
  & & & $p=11$ \\ 
  & & & $\DD^*(B)$ is a Hadamard \\
$M_{11}$ & $2$-$(12,6,5)$ & & 3-subdesign of the \\
 & & & Witt design $W_{12}$ \\
 & & & ($3$-$(12, 6, 2)$ design) \\ \hline
\end{tabular}
  \caption{Possibilities when $s=1$ in case (f).}
\label{tab1}
  \end{center}
\end{table}

\begin{table}
\begin{center}
  \begin{tabular}{l|l|c|l}
\hline
$G_B^{\Ga_{\BB}(B)}$ & $\DD^*(B)$ & $(v,b,r,\l)$ & Conditions \\  \hline
 & & & $n \geq 3$ odd \\
$\le \AGL(n, 3)$ & & $\pmat{\frac{(3^n-1)3^j}{2}\\ 3^n\\ 3^{n-j}(3^j-1)\\ \frac{(3^n-1)(3^j-2)}{2}+ \frac{3^{n-j}-1}{2}}$ & $p= \frac{3^n-1}{2}$ \\
 & & & $1 \le j \le n-1$ \\ \hline
 & & & $a$ an odd divisor \\
 & & & of $2p+1$ \\ 
$\le \PGL(n,2)$ & & $\pmat{a(2^{n-1} - 1)\\ 2^n -1\\ \frac{(2^{n} - 1)(a-1)}{a}\\ (2^{n-1} - 1)(a - 2) + \frac{2^{n} - 1-a}{2a}}$ & $3 \le a \le \frac{2p+1}{3}$ \\
 & & & $p=2^{n-1} - 1$ \\ 
 & & & a Mersenne prime \\
 & & & ($n - 1 \ge 3$ a prime) \\ \hline
$A_7$ & $\overline{\DD^*}(B) \cong \PG(3,2)$ & $(35,15,12,22)$ & \\ \hline
  \end{tabular}
  \caption{Possibilities when $s=2$ in case (f).}
\label{tab2}
  \end{center}
\end{table}

\begin{remark}
\label{rem:e}
{\em (1) In (e), denote by $s$ the valency of $\Ga_{\PP}[\hat{B}, \hat{C}]$ for adjacent $B, C \in \BB$, and by $t$ the number of blocks of $\PP$ contained in $C$ which contain at least one neighbour of a fixed vertex in $B \cap \Ga(C)$. Since $r_{\hat \BB}=a-1$, the parameters with respect to $\PP$ satisfy $b_{\PP}=(a-1)s$ and $r_{\PP}=(a-1)t$. Since $v_{\PP}r_{\PP}= b_{\PP}k_{\PP}$ and $v_{\PP}=p$, we have $pt=k_{\PP}s$. Since $1 \le t \le s \le a-1$, $1 \le k_{\PP} \le p$ and $p$ is a prime, we have either (i) $k_{\PP}=p$ and $s=t$, or (ii) $s=pc$ and $t = k_{\PP}c$ for some integer $c$ with $1 \le c \le \lfloor \frac{a-1}{p}\rfloor$. 

Since $v-2k+\l = 0$ in (e), examples of $(\Ga, G, \BB)$ in this case can be constructed using \cite[Construction 3.8]{lz} by first lifting a $(G, 2)$-arc transitive graph to a $G$-symmetric 3-arc graph and then lifting the latter to a $G$-symmetric graph $\Ga$ by the standard covering graph construction \cite{b}.

(2) The condition $(v,b,r,\lambda)=(pa,a,a-1,p(a-2))$ in (e) is sufficient for $\Ga_{\BB}$ to be $(G, 2)$-arc transitive. In fact, in this case for any $B \in \BB$ and $\a \in B$, there exists exactly one block $A \in \Ga_{\BB}(B)$ which contains no neighbour of $\a$. Thus, for any distinct $C, D \in \Ga_{\BB}(B) \setminus \{A\}$, there exist $\b \in C$ and $\g \in D$ which are adjacent to $\a$ in $\Ga$. Since $\Ga$ is $G$-symmetric, there exists $g \in G_{\a}$ such that $\b^g=\g$. So $(B,C)^g=(B,D)$. Since $g$ fixes $\a$, it must fix the unique block of $\Ga_{\BB}(B)$ having no neighbour of $\a$, that is, $A^g = A$. It follows that $G_{A,B}$ is transitive on $\Ga_{\BB}(B) \setminus  \{A\}$ and hence $\Ga_{\BB}$ is $(G,2)$-arc transitive. 

(3)  In (f), it seems challenging to determine $G_B^{\Ga_{\BB}(B)}$ and $\DD^*(B)$ when $s$ is not specified.    

We appreciate Yuqing Chen for constructing the following example for the third row of Table \ref{tab1}. 
Denote $F = \GF(2^n)$ and let $H$ be a subgroup of the additive group $E = (\GF(2^n), +)$ of order $2^{n-m}$ (where $2^n - 1$ is not necessarily a Mersenne prime). Then $E \rtimes F^*$ acts on $E$ as a 2-transitive subgroup of $\AGL(n,2)$. The incidence structure whose point set is $E$ and blocks are the complements in $E$ of the $E \rtimes F^*$-orbits of $H$ is a $2$-$(2^n, 2^n - 2^{n-m}, (2^m - 1)(2^n - 2^{n-m} - 1))$ design admitting $E \rtimes F^*$ as a 2-point transitive group of automorphisms. 
}
\end{remark} 
    
In the case when $k=v-3 \ge 1$ or $k=v-5 \ge 1$, Theorem \ref{thm:para} enables us to obtain necessary and sufficient conditions for $\Ga_{\BB}$ to be $(G, 2)$-arc transitive. This will be given in Theorem \ref{thm:main} and Corollary \ref{thm:main5} in Section \ref{sec:3}, respectively.  

We will use standard notation and terminology on block designs \cite{bjl,CD} and permutation groups \cite{Dixon-Mortimer}. The set of arcs of a graph $\Si$ is denoted by $\Arc(\Si)$.

\section{Proof of Theorem \ref{thm:para}}

\begin{proof}{\bf of Theorem \ref{thm:para}}~
Suppose $\Ga_{\BB}$ is $(G, 2)$-arc transitive. Then $G_B$ is 2-transitive on $\Ga_{\BB}(B)$ and hence $\l$ defined in (\ref{eq:l}) is independent of the choice of distinct $C, D \in \Ga_{\BB}(B)$. It is known \cite[Section 3]{lz} that either $\l = 0$ or $\DD^{*}(B)$ is a $2$-$(b, r, \l)$ design of $v$ `blocks' with $G_B$ doubly transitive on its points and transitive on its blocks and flags. Since $k=v-p \ge 1$, we have 
\begin{equation}
\label{eq:vr1}
vr=b(v-p)
\end{equation}
and by \cite[Corollary 3.3]{lz}, 
\begin{equation}
\label{eq:lambda1}
\lambda(b-1)=(v-p)(r-1).
\end{equation}

Consider the case $\l=0$ first. In this case we have $r=1$ as $v-p \ge 1$. Thus $v=b(v-p)$ and so $v=p+ \frac{p}{b-1}$. Since $v$ is an integer and $p$ is a prime, we have $b=p+1$ or $2$, and therefore $(v,b,r,\lambda)=(p+1,p+1,1,0)$ or $(2p,2,1,0)$. In the former case, we have $k=v-p=1$ and $\Ga \cong (|V(\Ga)|/2) \cdot K_2$. Moreover, the actions of $G_B$ on $B$ and $\Ga_{\BB}(B)$ are permutationally isomorphic. Thus $G_{(B)} = G_{[B]}$ and $G_B^B \cong G_B^{\Ga_{\BB}(B)}$ is 2-transitive of degree $p+1$. On the other hand, for any connected $(p+1)$-valent $(G,2)$-arc transitive graph $\Si$, define $\Ga$ to be the graph with vertex set $\Arc(\Si)$ and edges joining $(\s, \t)$ to $(\t, \s)$ for all $(\s, \t) \in \Arc(\Si)$. Then $\Ga$ is $G$-symmetric admitting $\BB = \{B(\s): \s \in V(\Si)\}$ (where $B(\s) = \{(\s, \t): \t \in \Si(\s)\}$) as a $G$-invariant partition such that $(v,b,r,\lambda)=(p+1,p+1,1,0)$ and $\Ga_{\BB} \cong \Si$. (This simple construction was used in \cite[Example 2.4]{mpz} for trivalent $\Si$. It is a very special case of the flag graph construction \cite[Theorem 4.3]{zhou2003}.)
In the case where $(v,b,r,\lambda)=(2p,2,1,0)$, $\Ga[B, C]$ is a bipartite $G_{B,C}$-edge transitive graph with $p$ vertices in each part of its bipartition, $\Ga \cong n \cdot \Ga[B, C]$ where $n=|V(\Ga)|/2p$, $\Ga_{\BB} \cong C_n$, and therefore $G/G_{(\BB)} = D_{2n}$. 

Assume $\l \ge 1$ from now on. Denote by $r^*$ the replication number of $\DD^*(B)$, that is, the number of `blocks' containing a fixed `point'. We distinguish between the following two cases.

\medskip
\textbf{Case 1}: $v$ is not a multiple of $p$.
 
In this case, $v$ and $v-p$ are coprime. Thus, by (\ref{eq:vr1}), $v$ divides $b$ and $v-p$ divides $r$. On the other hand, as noticed in \cite{lz}, by the well-known Fisher's inequality we have $b \leq v$ and $r \leq v-p$. Thus $v=b$ and $r=v-p=k$. From (\ref{eq:lambda1}) we then have $\l = \frac{(v-p)(v-p-1)}{v-1} = (v-2p)+\frac{p(p-1)}{v-1}$. Note that $v \ne p+1$, for otherwise $\l=0$, which contradicts our assumption $\l \ge 1$. Since $\l$ is an integer, $v-1$ is a divisor of $p(p-1)$. Since $p$ is a prime and $v-1 \ge p+1$, it follows that $p$ is a divisor of $v-1$. Set $a = \frac{v-1}{p}$. Then $a \ge 2$ is a divisor of $p-1$ and $(v,b,r,\lambda)=\left(pa+1, pa+1, p(a-1)+1, p(a-2)+\frac{p+a-1}{a}\right)$. Hence $\DD^*(B)$ is a 2-transitive symmetric 2-$\left(pa+1, p(a-1)+1, p(a-2)+\frac{p+a-1}{a}\right)$ design. Thus, by the classification of 2-transitive symmetric designs \cite{k} (see also \cite[Theorem XII-6.22]{bjl}), $\DD^*(B)$ or $\overline{\DD^*}(B)$ is isomorphic to one of the following: 
\begin{itemize}
\item $\PG_{n-1}(n, q)$ (where $n \ge 2$ and $q$ is a prime power);
\item the unique 2-$(11, 5, 2)$ design; 
\item the unique symmetric 2-$(176, 50, 14)$ design;  
\item the unique 2-$(2^{2m}, 2^{m-1}(2^{m} - 1), 2^{m-1}(2^{m-1} - 1))$ design (where $m \ge 2$). 
\end{itemize}
Since $\PG_{n-1}(n, q)$ has $\frac{q^{n+1}-1}{q-1}$ points and block size $\frac{q^{n}-1}{q-1}$, while $\DD^*(B)$ has $pa+1$ `points' and block size $p(a-1)+1$, by comparing these parameters one can show that $\DD^*(B) \not \cong \PG_{n-1}(n, q)$. In the same fashion we can see that none of the 2-transitive symmetric designs above can occur as $\DD^*(B)$. On the other hand, since $p$ is a prime, $\overline{\DD^*}(B)$ cannot be isomorphic to the unique symmetric 2-$(176, 50, 14)$ design or the unique 2-$(2^{2m}, 2^{m-1}(2^{m} - 1), 2^{m-1}(2^{m-1} - 1))$ design. We are left with the case where $\overline{\DD^*}(B) \cong \PG_{n-1}(n, q)$ or $\overline{\DD^*}(B)$ is isomorphic to the unique 2-$(11, 5, 2)$ design.

It is easy to verify that $\overline{\DD^*}(B) \cong \PG_{n-1}(n, q)$ only if $p = \frac{q^n - 1}{q-1}$ and $a = q$. In this case, $G_B^{\Ga_{\BB}(B)}$ is isomorphic to a 2-transitive subgroup of $\PGammaL(n+1,q)$ since $G_B^{\Ga_{\BB}(B)} \le \Aut(\overline{\DD^*}(B)) \cong \PGammaL(n+1,q)$. Moreover, we have $G_B^B \cong G_B^{\Ga_{\BB}(B)}$ since the actions of a 2-transitive subgroup of $\PGammaL(n+1,q)$ on the point set and the block set of $\PG_{n-1}(n, q)$ are permutationally isomorphic. Furthermore, if $g \in G_{(\BB)}$, then $g \in G_{(B)}$ since $\overline{\DD^*}(B)$ is self-dual. Since this holds for every $B \in \BB$, $g$ fixes every vertex of $\Ga$. Since $G \le \Aut(\Ga)$ is faithful on $V(\Ga)$, we conclude that $g=1$ and so $G$ is faithful on $\BB$. Therefore,  (c) occurs. 

$\overline{\DD^*}(B)$ is isomorphic to the unique 2-$(11, 5, 2)$ design if and only if $p=5$ and $a = 2$. In this case, since the automorphism group of this symmetric 2-$(11, 5, 2)$ design is $\PSL(2,11)$ (see e.g. \cite[Theorem IV.7.14]{bjl}), $G_B^B$ is isomorphic to a 2-transitive subgroup of $\PSL(2,11)$. Since $|G_B^{B}| \ge 11 \cdot 10$ but no proper subgroup of $\PSL(2,11)$ has order greater than 60, we have $G_B^B \cong G_B^{\Ga_{\BB}(B)} \cong \PSL(2,11)$ and (d) occurs.
 
\medskip
\textbf{Case 2}: $v=pa$ is a multiple of $p$, where $a \geq 2$ is an integer.  

In this case (\ref{eq:vr1}) becomes $ar=b(a-1)$. Thus $a$ divides $b$ and $a-1$ divides $r$. On the other hand, Fisher's inequality yields $b \le pa$ and $r \le p(a-1)$. So $b=at$ and $r=(a-1)t$ for some integer $t$ between $1$ and $p$. By (\ref{eq:lambda1}), $\l=\frac{p(a-1)((a-1)t-1)}{at-1} =p(a-2)+ \frac{p(t-1)}{at-1}$. 

\medskip
\textbf{Subcase 2.1}: $t=1$. 
Then $(v,b,r,\lambda)=(pa,a,a-1,p(a-2))$, where $a\geq 3$ as $\l \ge 1$ by our assumption. Thus, by \cite[Eq.(3)]{lz}, any two distinct `blocks' of $\overline{\DD}(B)$ intersect at $\overline{\l} = v-2k+\l = 0$ `points'. That is, the `blocks' $B \setminus \Ga(C)$ ($C \in \Ga_{\BB}(B)$) of $\overline{\DD}(B)$ are pairwise disjoint and hence \cite[Theorem 3.7]{lz} applies. Following \cite[Section 3]{lz}, define
$\PP= \cup_{B\in\BB} \{B\setminus\Ga(C):  C\in \Ga_{\BB}(B)\}$.
Then $\PP$ is a proper refinement of  $\BB$. Denote $\hat B = \{B\setminus\Ga(C)\in \PP:  C\in \Ga_{\BB}(B)\}$. Then $\hat \BB=\{\hat B :  B\in  \BB\}$ is a $G$-invariant partition of $\PP$. Denote by $v_{\hat \BB}$,  $k_{\hat \BB}$,  $b_{\hat \BB}$, $r_{\hat \BB}$ the parameters with respect to $(\Ga_{\PP}, \hat \BB)$. It can be verified (see \cite[Theorem 3.7]{lz}) that $(\Ga_{\PP})_{\hat \BB} \cong \Ga_{\BB}$, $v_{\hat \BB}=v/p=a$, $k_{\hat \BB}=v_{\hat \BB}-1=a-1$, $b_{\hat \BB}=b=a$, $r_{\hat \BB}= r =a-1$ and $\DD(\hat{B})$ has no repeated blocks. Thus, by \cite[Theorem 1]{lpz} (or \cite[Theorem 3.7]{lz}),  $\Ga_{\PP}\cong \Xi(\Ga_{\BB}, \Delta)$ for some self-paired $G$-orbit $\Delta$ on the set of 3-arcs of $\Ga_{\BB}$. Hence (e) occurs. 

\medskip
\textbf{Subcase 2.2}: $t \ge 2$. In this case, since $\l$ is an integer, $at-1$ is a divisor of $p(t-1)$. In particular, $at-1 \le p(t-1)$, which implies $a \le p-1$ and $\frac{p-1}{p-a} \le t \le p$. Since $at-1$ does not divide $t-1$ and $p$ is a prime, $at-1$ must be a multiple of $p$, say, $at-1 = ps$, so that $\frac{p(t-1)}{at-1} = \frac{t-1}{s}$ and $s$ divides $t-1$. Since $t = \frac{ps+1}{a}$ is an integer, $a$ is a divisor of $ps+1$. Therefore, $(v,b,r,\l) = \left(pa, ps+1, \frac{(ps+1)(a-1)}{a}, p(a-2) + \frac{ps-a+1}{as}\right)$. Since $\l$ is an integer, $s$ is a divisor of $\frac{ps-a+1}{a}$ and so $s$ is a divisor of $a-1$. This together with $\frac{p-1}{p-a} \le t = \frac{ps+1}{a}$ implies $\frac{a-1}{p-a} \le s \le a-1$. Therefore, case (f) occurs.   

The rest of the proof is devoted to the case $s=1$ in (f). In this case $\DD^*(B)$ is a 2-$\left(p+1, \frac{(p+1)(a-1)}{a}, p(a-2) + \frac{p-a+1}{a}\right)$ design with $pa$ `blocks' such that each `point' is contained in exactly $r^* = p(a-1)$ `blocks'. Moreover, $\DD^*(B)$ admits $G_B$ as a group of automorphisms acting 2-transitively on the set $\Ga_{\BB}(B)$ of $p+1$ `points'. All 2-transitive groups are known (see e.g. \cite[Section 7.7]{Dixon-Mortimer}). First, since $p+1 \ne q^3 + 1, q^2 + 1$ for any prime power $q$, $G_B^{\Ga_{\BB}(B)}$ cannot be a unitary, Suzuki or Ree group. Since $r < p+1$, $S_{p+1}$ is $r$-transitive on $p+1$ points (in its natural action) but on the other hand $\DD^*(B)$ has $pa < \choose{p+1}{r}$ blocks. Hence $G_B^{\Ga_{\BB}(B)} \not \cong S_{p+1}$. Similarly, as $A_{p+1}$ is $(p-1)$-transitive in its natural action, $G_B^{\Ga_{\BB}(B)} \not \cong A_{p+1}$ unless $r = p-1$. In this exceptional case, $\DD^*(B)$ has $pa = \choose{p+1}{2}$ blocks and so is isomorphic to the complementary design of the trivial design $K_{p+1}$. This gives the second row in Table \ref{tab1}. 

If $G_B^{\Ga_{\BB}(B)}$ is affine, then $p+1 = q^n$ for some prime power $q$ and integer $n \ge 1$, which occurs if and only if $q=2$ and $p=2^n - 1$ is a prime.
In this case, $n$ must be a prime and $p=2^n - 1$ is a Mersenne prime, and $G_B^{\Ga_{\BB}(B)}$ is isomorphic to a 2-transitive subgroup of $\AGL(n,2)$. Moreover, $a = 2^m$ for some integer $1 \le m \le n-1$, and so $(v,b,r,\l,r^*) = (2^m (2^n - 1), 2^n, 2^n - 2^{n-m}, (2^m - 1)(2^n - 2^{n-m} - 1), (2^n - 1)(2^m - 1))$. This gives the third row in Table \ref{tab1}. 
%When $m=1$, $(v,b,r,\l,r^*) = (2(2^{n} - 1), 2^n, 2^{n-1}, 2^{n-1} - 1, 2^n - 1)$ and so 
%$\DD^*(B)$ is a $2$-$(2^n, 2^{n-1}, 2^{n-1} - 1)$ design. Such a design exists 
%(due to Yuqing Chen)  --- take the points as the elements of an elementary 
%$2$-group $E$ and take the blocks as all cosets of a subgroup of $E$ of index 2.

If $G_B^{\Ga_{\BB}(B)}$ is projective, then $p+1 = \frac{q^n - 1}{q-1}$ for a prime power $q$ and an integer $n \ge 2$. Thus $n=2, p = q$ and $G_B^{\Ga_{\BB}(B)}$ is isomorphic to a 2-transitive subgroup of $\PGL(2, p)$. Since $G_B$ is transitive on the $p(p+1)(a-1)$ flags of $\DD(B)$, $p(p+1)(a-1)$ is a divisor of $|\PGL(2, p)| = (p-1)p(p+1)$ and so $a-1$ is a divisor of $p-1$. Note that the second 2-transitive action of $A_5 \cong \PSL(2,5)$ with degree 6 is covered here, and that of $A_6 \cong \PSL(2,9)$ with degree 10, of $A_7$ with degree 15, and of $A_8 \cong \PSL(4,2)$ with degree 15 cannot happen since 9 and 14 are not prime. Similarly, the second 2-transitive action of $\PSL(2,8) \le \Sp_{6}(2)$ of degree 28 and that of $\PSL(2,11) \le M_{11}$ of degree 11 cannot happen. So we have the fourth row in Table \ref{tab1}. 

If $G_B^{\Ga_{\BB}(B)}$ is symplectic, then $p+1 = 2^{m-1}(2^m \pm 1)$ for some $m \geq 2$. If $p+1 = 2^{m-1}(2^m + 1)$, then $p = (2^{m-1}+1)(2^m-1)$, which cannot happen since $p$ is a prime. Similarly, if $p+1 = 2^{m-1}(2^m - 1)$, then $p = (2^{m-1}-1)(2^m+1)$, which occurs if and only if $m=2$ and $p=5$. In this exceptional case, we have $G_B^{\Ga_{\BB}(B)} \cong \Sp_{4}(2)$ ($\cong S_6$), $a=2$ or 3, and hence $(v,b,r,\l, r^*) = (10, 6, 3, 2, 5)$ or $(15, 6, 4, 6, 10)$. The latter cannot happen since a 2-$(6,4,6)$ design does not exist \cite[Table A1.1]{bjl}. Thus $\DD^*(B)$ is isomorphic to the unique $2$-$(6,3,2)$ design. This gives the fifth row in Table \ref{tab1}. 
 
By comparing the degree $p+1$ of $G_B^{\Ga_{\BB}(B)}$ with that of the ten sporadic 2-transitive groups \cite[Section 7.7]{Dixon-Mortimer}, one can verify that among such groups only the following may be isomorphic to $G_B^{\Ga_{\BB}(B)}$: 
$M_{11}$ (degree $p+1=12$);
$M_{12}$ (degree $p+1=12$);
$M_{24}$ (degree $p+1=24$). 

In the case of $M_{24}$, $a$ is 2, 3, 4, 6, 8 or 12, and so $(v,b,r,\l, r^*) = (46,24,12,11,23)$, $(69,24,16,30,46)$, $(92,24,18,51,69)$, $(138,24,20,95,115)$, $(184,24,21,140,161)$ or $(276,24,22,$ $231,253)$. It is well known \cite[Chapter IV]{bjl} that $M_{24}$ is the automorphic group of the unique $5$-$(24, 8, 1)$ design (the Witt design $W_{24}$) which is also a $2$-$(24, 8, 77)$ design, and that up to isomorphism the natural action of $M_{24}$ on the points of $W_{24}$ is the only 2-transitive action of $M_{24}$ with degree 24. Hence $G_B^{\Ga_{\BB}(B)} \not \cong M_{24}$.

In the cases of $M_{11}$ and $M_{12}$, $a$ is 2, 3, 4 or 6, and so $(v,b,r,\l, r^*) = (22,12,6,5,11)$, $(33,12,8,14,22)$, $(44,12,9,24,33)$ or $(66,12,10,45,55)$. Since by \cite[Section II.1.3]{CD} a 2-$(12, 8, 14)$ or 2-$(12, 9, 24)$ design does not exist, the second and third possibilities can be eliminated. Thus, if $G_B^{\Ga_{\BB}(B)} \cong M_{11}$ or $M_{12}$, then $\DD^*(B)$ is isomorphic to a 2-$(12,6,5)$ or 2-$(12,10,45)$ design. It is well-known \cite[Chapter IV]{bjl} that $M_{12}$ is the automorphic group of the unique $5$-$(12, 6, 1)$ design (the Witt design $W_{12}$) which is also a $2$-$(12, 6, 30)$ design. Since up to isomorphism the natural action of $M_{12}$ on the points of $W_{12}$ is the only 2-transitive action of $M_{12}$ with degree 12, we have $G_B^{\Ga_{\BB}(B)} \not \cong M_{12}$. $M_{11}$ is the automorphic group of a $3$-$(12, 6, 2)$ design (that is, a Hadamard 3-subdesign of $W_{12}$ \cite[Chapter IV]{bjl}) which is also a 2-$(12,6,5)$ design. Since up to isomorphism the natural action of $M_{11}$ on the points of this design is the only 2-transitive action of $M_{11}$ with degree $12$, we conclude that if $G_B^{\Ga_{\BB}(B)} \cong M_{11}$ then $\DD^*(B)$ is isomorphic to this 2-$(12,6,5)$ design. 
This gives the last row in Table \ref{tab1}.

In the same fashion one can prove that, if $s=2$ in (f), then we have possibilities in Table \ref{tab2}. 
\qed
\end{proof}

\section{$p=3, 5$}
\label{sec:3}

Theorem \ref{thm:para} provides a necessary condition for $\Ga_{\BB}$ to be $(G,2)$-arc transitive when $k=v-p$ for any prime $p \ge 3$. This condition may be sufficient for some special primes $p$, and in this section we prove that this is the case when $p=3$ or $5$. Moreover, when $p=3$ we obtain more structural information about $\Ga$ (see Theorem \ref{thm:main} below). In particular, in the last case in Theorem \ref{thm:main} (which corresponds to case (f) in Theorem \ref{thm:para}), $\Ga$ can be constructed from $\Ga_{\BB}$ by using a simple construction introduced in \cite[Section 4.1]{lz}. Given a regular graph $\Si$ with valency at least $2$ and a self-paired subset $\Delta$ of the set of 3-arcs of $\Si$, define \cite{lz} $\Ga_2(\Sigma, \Delta)$ to be the graph with the set of {\em $2$-paths} (paths of length 2) of $\Si$ as vertex set such that two distinct ``vertices" $\t\s\t'$ ($=\t'\s\t$) and $\eta \ve \eta'$ ($=\eta' \ve \eta$) are adjacent if and only if they have a common edge (that is, $\s \in \{\eta, \eta'\}$ and $\ve \in \{\t, \t'\}$) and moreover the two $3$-arcs
(which are reverses of each other) formed by ``gluing" the common edge are in $\De$. (As noted in \cite{lz}, when $\De$ is the set of all 3-arcs of $\Si$, $\Ga_2(\Sigma, \Delta)$ is exactly the path graph $P_3(\Si)$ introduced in \cite{bh}.)

In the proof of Theorem \ref{thm:main} we will use the following Lemma.  

\begin{lemma} 
\label{le1}
Let $\Ga$ be a $G$-symmetric graph that admits a nontrivial $G$-invariant partition $\BB$ such that $k=v-i$, where $i\geq 1$. Then the multiplicity $m$ of $\DD(B)$ and $\overline{\DD}(B)$ is a common divisor of $r$ and $b$. 
\end{lemma} 

\begin{proof}
As in \cite{lz}, we may view $\DD(B)$ and $\overline{\DD}(B)$ as hypergraphs with vertex set $B$ and hyperedges $\Ga(C) \cap B$ and $B \setminus \Ga(C)$, $C \in \Ga_{\BB}(B)$, respectively, with each hyperedge repeated $m$ times. It is easy to see that as hypergraphs they have valencies $\val(\DD(B))=r=b-\frac{ib}{v}$ and $\val(\overline{\DD}(B))=b-r$, respectively. Since $m$ is a divisor of each of these valencies, it must be a common divisor of $r$ and $b$.
\qed
\end{proof}

Denote by $K_{n,n}$ the complete bipartite graph with $n$ vertices in each part of its bipartition, and by $\Si_1 - \Si_2$ the graph obtained from a graph $\Si_1$ by deleting the edges of a spanning subgraph $\Si_2$ of $\Si_1$. Denote by $G_{B, C}$ the subgroup of $G$ fixing $B$ and $C$ setwise. A few statements in the following theorem are carried over directly from Theorem \ref{thm:para}, and we keep them there for the completeness of the result. 
 
\begin{theorem} 
\label{thm:main}
Let $\Ga$ be a $G$-symmetric graph with $V(\Ga)$ admitting a nontrivial $G$-invariant partition $\BB$ such that $k=v-3 \geq 1$ and $\Ga_{\BB}$ is connected of valency $b \geq 2$, where $G \le \Aut(\Ga)$.  Then $\Ga_{\BB}$ is $(G, 2)$-arc transitive if and only if one of the following holds: 
\begin{itemize}
\item[\rm (a)] $(v,b,r,\lambda)=(4,4,1,0)$  and $G_B^B \cong A_4$ or $S_4$; 
\item[\rm (b)] $(v,b,r,\lambda)=(6,2,1,0)$ and $\Ga_{\BB} \cong C_n$, where $n = |V(\Ga)|/6$;
\item[\rm (c)] $(v,b,r,\lambda)=(7,7,4,2)$ and $G_B^B \cong \PSL(3,2)$;
\item[\rm (d)]  $(v,b,r,\lambda)=(3a,a,a-1,3a-6)$ for some integer $a \geq 3$;
\item[\rm (e)] $(v,b,r,\lambda)=(6,4,2,1)$ and $G_B^{\Ga_{\BB}(B)} \cong A_4$ or $S_4$.
\end{itemize}
Moreover, in (a) we have $G_B^{\Ga_{\BB}(B)} \cong A_4$ or $S_4$, $\Ga \cong (|V(\Ga)|/2) \cdot K_2$, and every connected 4-valent 2-arc transitive graph can occur as $\Ga_{\BB}$ in (a). 

In (b), we have $\Ga \cong 3n \cdot K_2, n \cdot C_6$ or $n \cdot K_{3,3}$, and $G/G_{(\BB)} = D_{2n}$.

In (c), $\overline{\DD}(B)$ is isomorphic to the Fano plane $\PG(2,2)$, $G_B^{\Ga_{\BB}(B)} \cong \PSL(3,2)$, $G$ is faithful on $\BB$, and $\Ga[B,C] \cong 4 \cdot K_2$, $K_{4,4} - 4 \cdot K_2$ or $K_{4,4}$. In the first case $\Ga$ is $(G,2)$-arc transitive, and in the last two cases $\Ga$ is connected of valency 12 and 16 respectively.  

In (d), the statements in (e) of Theorem \ref{thm:para} hold with $p=3$.

In (e), we have $\Ga \cong \Ga_2(\Ga_{\BB}, \De)$ for a self-paired $G$-orbit $\De$ on 3-arcs of $\Ga_{\BB}$, and every connected 4-valent $(G, 2)$-arc transitive graph can occur as $\Ga_{\BB}$ in (e).
\end{theorem}

\begin{proof}  
\textit{Necessity}~~Suppose $\Ga_{\BB}$ is $(G, 2)$-arc transitive. Since $p=3$, by Theorem \ref{thm:para}, $(v,b,r,\lambda)$ is one of the following: 
\begin{itemize}
\item[] (a) $(4,4,1,0)$;\quad (b) $(6,2,1,0)$;\quad (c) $(7, 7, 4, 2)$ (for which $n=q=2$); 
\item[] (d) $(3a,a,a-1,3a-6)$ (where $a \geq 3$);\quad (e) $(6,4,2,1)$ (for which $a=2$ and $s=1$). 
\end{itemize}

In case (a), $G_B^B \cong G_B^{\Ga_{\BB}(B)}$ is 2-transitive of degree $4$, and in case (b), $\Ga[B, C] \cong 3 \cdot K_2, C_6$ or $K_{3,3}$ for adjacent $B, C \in \BB$. The properties for cases (a), (b) and (d) follow from Theorem \ref{thm:para} immediately. 

\medskip
\textbf{Case (c)}:~~In this case $\DD(B)$ is the biplane of order $2$. In other words, $\overline{\DD}(B)$ is isomorphic to the Fano plane $\PG(2,2)$. Since $G_B^B$ induces a group of automorphisms of the self-dual $\overline{\DD}(B)$, we have $G_B^B \le \Aut(\overline{\DD}(B)) \cong \PSL(3,2)$ and $G_B^{\Ga_{\BB}(B)} \le \Aut(\overline{\DD}(B))$. Since $G_B$ is 2-transitive on $\Ga_{\BB}(B)$ of degree 7, we have $|G_B^{\Ga_{\BB}(B)}| \ge 7 \cdot 6 = 42$. Since no proper subgroup of $\PSL(3,2)$ has order greater than 24, it follows that $G_B^{\Ga_{\BB}(B)} \cong \PSL(3,2)$. Since the actions of an automorphism group of $\PG(2,2)$ on the set of points and the set of lines are permutationally isomorphic, we have $G_B^B \cong \PSL(3,2)$. By Theorem \ref{thm:para}, $G$ is faithful on $\BB$.

We now prove $\Ga[B, C] \not \cong 2 \cdot C_{4}, C_{8}$, and if $\Ga[B,C] \cong 4 \cdot K_2$ then $\Ga$ is $(G, 2)$-arc transitive. Denote $A \cap \Ga(B)=\{u_1,u_2,u_3,u_4\}$ and $B \cap \Ga(A) = \{v_1,v_2,v_3,v_4\}$ for a fixed $A \in  \Ga_{\BB}(B)$.

Suppose $\Ga[A, B] \cong 2\cdot C_{4}$. Without loss of generality we may assume that each of $\{u_1, u_2, v_1,v_2\}$ and $\{u_3, v_3, v_3,v_4\}$ induces a copy of $C_4$ in $\Ga$. Since $\l=2$, $|B \cap \Ga(A) \cap \Ga(F)| = 2$ for each $F \in \Ga_{\BB}(B) \setminus \{A\}$. Since there are exactly six such blocks $F$, and since $|B \cap \Ga(A)| = 4$ and the multiplicity of $\DD(B)$ is one, each pair $\{v_i, v_j\}$ $(1 \le i < j \le 4$) is equal to exactly one $B \cap \Ga(A) \cap \Ga(F)$. So there exist $C,D \in \Ga_{\BB}(B) \setminus \{A\}$ such that $B \cap \Ga(A) \cap \Ga(C)=\{v_1,v_2\}$ and $B \cap \Ga(A) \cap \Ga(D)=\{v_1,v_3\}$. Since $\Ga_{\BB}$ is $(G, 2)$-arc transitive, there exists $g\in G$ such that $(A,B,C)^g=(A,B,D)$. Hence $(B \cap \Ga(A) \cap \Ga(C))^g = B \cap \Ga(A) \cap \Ga(D)$, that is, $\{v_1,v_2\}^g =\{v_1,v_3\}$. However, since $g \in G_{A, B}$, it permutes the two cycles of $\Ga[A,B]$ and so $\{v_1,v_2\}^g = \{v_1,v_2\}$ or $\{v_3,v_4\}$, which is a contradiction. 
 
Suppose $\Ga[A,B]  \cong C_{8}$. Without loss of generality we may assume that $\Ga[A,B]$ is the cycle $(v_1,u_1, v_2, u_2,v_3,u_3,v_4,u_4, v_1)$. As above there exists $C \in \Ga_{\BB}(B) \setminus \{A\}$ such that $B \cap \Ga(A) \cap \Ga(C)=\{v_1,v_2\}$. Since $r=4$, there exist distinct $D, F \in \Ga_{\BB}(B) \setminus \{A, C\}$ such that $v_1\in B \cap \Ga(D) \cap \Ga(F)$. Since $\l=2$, either $B \cap \Ga(A) \cap \Ga(D)$ or $B \cap \Ga(A) \cap \Ga(F)$ is equal to $\{v_1,v_3\}$, say, $B \cap \Ga(A) \cap \Ga(D) = \{v_1,v_3\}$. Since $\Ga_{\BB}$ is $(G, 2)$-arc transitive, there exists $g\in G$ such that $(A,B,C)^g=(A,B,D)$. Hence $\{v_1,v_2\}^g =\{v_1,v_3\}$. However, $g \in G_{A, B}$ induces an automorphism of $\Ga[A,B]$. On the other hand, the distances from $v_1$ to $v_2$ and $v_3$ in $\Ga[A,B]$ are 2 and 4, respectively, and this is contradiction. 

So far we have proved that $\Ga[A, B] \not \cong 2\cdot C_{4}, C_8$. Since $k = 4$ and $\Ga[B,C]$ is $G_{B, C}$-edge transitive, we must have $\Ga[B,C] \cong 4 \cdot K_2, K_{4,4} - 4\cdot K_2$ or $K_{4,4}$. Suppose $\Ga[B,C]\cong 4 \cdot K_2$. Then for $\a \in B$ the action of $G_{\a}$ on $\Ga(\a)$ and $\Ga_{\BB}(\a)$ are permutationally isomorphic. Note that $\Ga_{\BB}(\a)$ is a block of $\DD^*(B) \cong \DD(B)$. Since $G_{B}^{B} \cong G_{B}^{\Ga_{\BB}(B)} \cong \PSL(2,7) \cong \Aut(\DD^*(B))$, the setwise stabilizer of $\Ga_{\BB}(\a)$ in $G_{B}^{\Ga_{\BB}(B)}$ is isomorphic to $S_4$ and hence is 2-transitive on $\Ga_{\BB}(\a)$ as $|\Ga_{\BB}(\a)| = 4$. One can verify that this stabilizer is equal to $G_{\a}$. Thus $G_{\a}$ is 2-transitive on $\Ga_{\BB}(\a)$ and so 2-transitive on $\Ga(\a)$. In other words, $\Ga$ is $(G, 2)$-arc transitive when $\Ga[B,C] \cong 4 \cdot K_2$. In the case where $\Ga[B,C] \cong K_{4,4} - 4\cdot K_2$ or $K_{4,4}$, since $\Ga_{\BB}$ is connected and $\overline{\DD}(B) \cong \PG(2,2)$, one can easily see that $\Ga$ is connected of valency 12 or 16 respectively.

\medskip
\textbf{Case (e)}:~~Since $(v,b,r,\lambda) = (6,4,2,1)$, $\DD^*(B)$ is the 2-$(4,2,1)$ design, that is, the complete graph on four vertices. This case coincides with the case $(v,k)=(6,3)$ in \cite[Theorem 4.1(b)]{jlw} and we have $G_B^{\Ga_{\BB}(B)} \cong A_4$ or $S_4$ since $G_B$ is 2-transitive on $\Ga_{\BB}(B)$ of degree 4. Since $(\l, r) = (1, 2)$, by \cite[Theorem 4.3]{lz} we have $\Ga \cong \Ga_2(\Ga_{\BB}, \De)$ for some self-paired $G$-orbit $\De$ on 3-arcs of $\Ga_{\BB}$. Moreover, by \cite[Theorem 4.3]{lz} again, for any connected 4-valent $(G, 2)$-arc transitive graph $\Si$ and any self-paired $G$-orbit $\De$ on 3-arcs of $\Si$, $\Ga = \Ga_2(\Si, \De)$ is a $G$-symmetric graph admitting $\BB_2 =\{B_2(\s): \s \in V(\Si)\}$ as a $G$-invariant partition such that $\Ga_{\BB_2} \cong \Si$ and the corresponding parameters are $(v,b,r,\lambda) = (6,4,2,1)$ and $k=v-3=3$, where $B_2(\s)$ is the set of 2-paths of $\Si$ with middle vertex $\s$. Since $\Si$ is $(G, 2)$-arc transitive with even valency, by \cite[Remark 4(c)]{lpz} such a $\De$ exists and hence $\Si$ can occur as $\Ga_{\BB}$ in (e).  
  
\textit{Sufficiency}~~  
We now prove that each of (a)--(e) implies that  $\Ga_{\BB}$ is $(G, 2)$-arc transitive. Since by Lemma \ref{le1} the multiplicity $m$ of $\DD(B)$ is a common divisor of $b$ and $r$, in cases (a)--(d) we have $m = 1$. In case (e), since $b=4$ and $\l \ge 1$, we have $m=1$ as well.  

In case (a), since $(v, b, r, \l) = (4, 4, 1, 0)$ and $k=1$, each vertex in $B$ has a neighbour in a unique block of $\Ga_{\BB}(B)$, yielding a bijection from $B$ to $\Ga_{\BB}(B)$. Using this bijection, one can see that the actions of $G_B$ on $B$ and $\Ga_{\BB}(B)$ are permutationally isomorphic. Since $G_B^B \cong A_4$ or $S_4$, $G_B$ is 2-transitive on $\Ga_{\BB}(B)$ and therefore $\Ga_{\BB}$ is $(G, 2)$-arc transitive. 

In case (b), since $\Ga_{\BB}$ is a cycle and is $G$-symmetric, it must be $(G, 2)$-arc transitive. 

In case (c), since $(v,b,r,\lambda) = (7,7,4,2)$, $\overline{\DD}(B) \cong \PG(2,2)$. Since $G_B^B \cong \PSL(3,2)$ and the actions of $\PSL(3,2)$ on the set of points and the set of lines of $\PG(2,2)$ are permutationally isomorphic, we have $G_B^{\Ga_{\BB}(B)} \cong \PSL(3,2)$. Since $\PSL(3,2)$ is 2-transitive on the set of lines of $\PG(2,2)$, $G_B$ is 2-transitive on $\Ga_{\BB}(B)$ and so $\Ga_{\BB}$ is $(G, 2)$-arc transitive.   
 
As shown in Remark \ref{rem:e}(2), in case (d), $\Ga_{\BB}$ is $(G,2)$-arc transitive.

In case (e), since $G_B^{\Ga_{\BB}(B)} \cong A_4$ or $S_4$ and $b=4$, $G_B$ is 2-transitive on $\Ga_{\BB}(B)$ and so $\Ga_{\BB}$ is $(G,2)$-arc transitive. 
\qed  
\end{proof}  

The following result about the case $p=5$ is largely a corollary of Theorem \ref{thm:para} (and Remark \ref{rem:e}(2)). So we omit its proof. 
  
\begin{corollary} 
\label{thm:main5}
Let $\Ga$ be a $G$-symmetric graph with $V(\Ga)$ admitting a nontrivial $G$-invariant partition $\BB$ such that $k=v-5 \geq 1$ and $\Ga_{\BB}$ is connected of valency $b \geq 2$, where $G \le \Aut(\Ga)$.  Then $\Ga_{\BB}$ is $(G, 2)$-arc transitive if and only if one of the following holds: 
\begin{itemize}
\item[\rm (a)] $(v,b,r,\lambda)=(6,6,1,0)$  and $G_B^B \cong G_B^{\Ga_{\BB}(B)} \cong A_6$ or $S_6$; 
\item[\rm (b)] $(v,b,r,\lambda)=(10, 2,1,0)$, $\Ga_{\BB} \cong C_n$ and $G/G_{(\BB)} = D_{2n}$, where $n = |V(\Ga)|/10$;
\item[\rm (c)] $(v,b,r,\lambda)=(21,21,16,12)$, $\overline{\DD^*}(B) \cong \PG(2, 4)$, $G_B^{B}  \cong G_B^{\Ga_{\BB}(B)}$ is isomorphic to a 2-transitive subgroup of $\PGammaL(3,4)$, and $G$ is faithful on $\BB$;
\item[\rm (d)] $(v,b,r,\lambda)=(11,11,6,3)$, $\overline{\DD^*}(B)$ is isomorphic to the unique $2$-$(11, 5, 2)$ design and $G_B^{B} \cong G_B^{\Ga_{\BB}(B)} \cong \PSL(2,11)$;
\item[\rm (e)] $(v,b,r,\lambda)=(5a,a,a-1,5a-10)$ for some integer $a \geq 3$;
\item[\rm (f)] either (1) $(v,b,r,\lambda)=(10,6,3,2)$, $\DD^*(B)$ is isomorphic to the unique $2$-$(6,3,2)$ design, and $G_B^{\Ga_{\BB}(B)} \cong \Sp_{4}(2)$ or $\PSL(2,5)$; or (2) $(v,b,r,\lambda)=(15,6,4,6)$, $\DD^*(B)$ is isomorphic to the complementary design of $K_6$ and $G_B^{\Ga_{\BB}(B)} \cong A_6$; or (3) $(v,b,r,\lambda)=(20,16,12,11)$, $\overline{\DD^*}(B) \cong \AG(2,4)$ and $G_B^{\Ga_{\BB}(B)}$ is isomorphic to a 2-transitive subgroup of $\AGammaL(2,4)$.  
\end{itemize}
\end{corollary}

As in Theorem \ref{thm:para}, in (a) above we have $\Ga \cong (|V(\Ga)|/2) \cdot K_2$ and every connected 6-valent 2-arc transitive graph can occur as $\Ga_{\BB}$ in (a). In (b), since $\Ga[B, C]$ is $G_{B,C}$-edge transitive, we have $\Ga \cong 5n \cdot K_2$, $n \cdot C_{10}$, $n \cdot (K_{5,5}-C_{10})$, $n \cdot (K_{5,5}- 5 \cdot K_2)$ or $n \cdot K_{5, 5}$. In (e) above, the same statements as in case (e) of Theorem \ref{thm:para} hold with $p=5$. The three cases in (f) arise because $(a, s) = (2, 1), (3, 1), (4, 3)$ are the only pairs satisfying the conditions in (f) of Theorem \ref{thm:para}. In (2) of (f), $G_B^{\Ga_{\BB}(B)}$ cannot be $\PGL(2,5)$ since the latter has no transitive action of degree 15. Similarly, in (1) of (f), $G_B^{\Ga_{\BB}(B)} \not \cong \PGL(2,5)$ because the $2$-$(6,3,2)$ design has 10 blocks of size 3 and $\PGL(2,5)$ is (sharply) 3-transitive of degree 6. The result in (3) of (f) follows because $\overline{\DD^*}(B)$ is a $2$-$(16,4,1)$ design in this case and $\AG(4,2)$ is the unique $2$-$(16,4,1)$ design \cite[Section 1.3]{CD}.

\vskip 1pc 
\noindent {\bf Acknowledgements}~~We would like to thank Professor Cheryl E. Praeger for helpful discussions on Mathieu groups, Professor Yuqing Chen for providing the example in Remark \ref{rem:e}(3), and an anonymous referee for helpful comments. Guangjun Xu was supported by the MIFRS and SFS scholarships of the University of Melbourne. Sanming Zhou was supported by a Future Fellowship (FT110100629) of the Australian Research Council.

\end{document}